\documentclass{amsart}
\usepackage{graphicx, tikz}
\usepackage{epsfig}
\usepackage{amssymb, amsfonts, latexsym, mathrsfs, amsmath, amsthm, graphicx, bbm, amstext}
\usepackage{txfonts,bbding,wasysym,pifont,stmaryrd,tipa}
\usepackage[all,cmtip]{xy}
\usepackage{float, mathtools}
\usepackage[utf8]{inputenc}
\usepackage{hyperref}
\usepackage{cleveref}
\crefname{section}{§}{§§}
\Crefname{section}{§}{§§}
\usetikzlibrary{arrows,positioning,shapes.geometric,matrix}
\DeclarePairedDelimiter\floor{\lfloor}{\rfloor}

\newtheorem{theorem}{Theorem}[section]
\newtheorem{lemma}[theorem]{Lemma}
\newtheorem{proposition}[theorem]{Proposition}
\newtheorem{corollary}[theorem]{Corollary}
\newtheorem{conj}[theorem]{Conjecture}

\theoremstyle{definition}
\newtheorem{definition}[theorem]{Definition}

\theoremstyle{remark}
\newtheorem{remark}[theorem]{Remark}
\newtheorem{notation}[theorem]{Notation}
\numberwithin{equation}{section}

\newcommand{\Alt}{\text{Alt}}
\newcommand{\Sym}{\text{Sym}}
\newcommand{\ord}{\text{ord}}
\newcommand{\Supp}{\text{Supp}}
\newcommand{\Aut}{\text{Aut}}
\newcommand{\Gal}{\text{Gal}}

\DeclareMathAlphabet{\mathpzc}{OT1}{pzc}{m}{it}
\DeclareMathOperator*{\spec}{Spec} 
 
\DeclareMathOperator*{\Ind}{Ind}

\newcommand{\PP}{\mathbb{P}}

\newcommand{\Z}{\mathbb{Z}} \newcommand{\A}{\mathbb{A}}

\begin{document}

\title{On The Inertia Conjecture for Alternating group Covers}
\author{Soumyadip Das}
\address{Statistics and Mathematics Unit,
Indian Statistical Institute, Bangalore Center, Bangalore 560059.}
\email{ soumyadip\_rs@isibang.ac.in}
\author{Manish Kumar}
\address{Statistics and Mathematics Unit,
Indian Statistical Institute, Bangalore Center, Bangalore 560059.}
\email{ manish@isibang.ac.in}


\keywords{Galois Cover, Inertia Conjecture, Ramification}

\begin{abstract}
The wild part of Abhyankar's Inertia Conjecture for a product of certain Alternating groups is shown for any algebraically closed field of odd characteristic. For $d$ a multiple of the characteristic of the base field, a new \'etale $A_d$-cover of the affine line is obtained using an explicit equation and it is shown that it has the minimal possible upper jump.
\end{abstract}

\maketitle


\section{Introduction}
Let $k$ be an algebraically closed field of characteristic $p>0$ and let $G$ be a finite group. The quasi $p$-subgroup $p(G)$ is the subgroup generated by all the Sylow $p$-subgroups of $G$. Let $\phi \colon Y \to \PP^1_k$ be a connected $G$-Galois cover of the projective line branched only at $\infty$. Then the induced $G/p(G)$-Galois cover is a prime-to-$p$ Galois cover of $\PP^1$ branched only at $\infty$ and hence it is trivial. So $G = p(G)$. Such a group is called a quasi $p$-group. Abhyankar's $1957$ conjecture (\cite{6}) on the affine $k$-curves implies that a finite group $G$ occurs as the Galois group of a cover of the projective line branched only at $\infty$ if and only if $G$ is a quasi $p$-group. This conjecture for the affine line was proved by Serre ($G$ solvable, \cite{4}) and Raynaud (\cite{3}).

Abhyankar in his $2001$ paper (\cite[Section 16]{19}) conjectured a more refined statement on the Galois \'{e}tale covers of the affine line which is known as the inertia conjecture. Let $G$ be a quasi $p$-group and $\phi \colon Y \rightarrow \mathbb{P}^1$ be a $G$-Galois cover of the projective line branched only at infinity. Then for any point $y$ of $Y$ in the fibre $\phi^{-1}(\infty)$, the inertia group $I$ at $y$ is an extension of a cyclic prime-to-$p$ group by a $p$-group $P$ (\cite[Chapter IV, Corollary 4]{Serre_loc}). Let $J$ be the subgroup of $G$ generated by all the conjugates of $P$ in $G$. Then $J$ is a normal subgroup of $G$ and $Y \to \PP^1$ induces a $G/J$-Galois cover of $\mathbb{P}^1$ which is at most tamely ramified at $\infty$. Since the tame fundamental group of $\mathbb{A}^1$ is trivial, $G=J$. Abhyankar's inertia conjecture says that the converse also holds.

The inertia conjecture can be compartmentalized into two parts (cf. \cite[Conjecture 4.2]{survey_paper}).

\begin{conj}\label{conj_break} Let $G$ be a finite quasi $p$-group.
\begin{description}
\item[Wild Part] A $p$-subgroup $P$ of $G$ occurs as the inertia group of a $G$-Galois cover of $\mathbb{P}^1$ branched only at $\infty$ if and only if the conjugates of $P$ generate $G$.
\item[Tame Part] Assume that a $p$-subgroup $P$ of $G$ occurs as the inertia group of a $G$-Galois \'{e}tale cover of the affine line. Let $\beta$ be an element of $G$ of prime-to-$p$ order contained in the normalizer of $P$ in $G$. Let $I$ be an extension of $\langle \beta \rangle$ by $P$ in $G$. Then there is a $G$-Galois cover $\phi \colon Y \to \mathbb{P}^1$ branched only at $\infty$ with inertia group $I$ at a point in $Y$ over $\infty$.
\end{description}
\end{conj}

By Abhyankar's Lemma (\cite[XIII, Proposition 5.2]{10}), one can reduce the tame part of the inertia group. On the other hand, Harbater has shown (\cite[Theorem 2]{2}) that the wild part of the inertia group can be increased. As a consequence, the inertia conjecture is known for Sylow $p$-subgroups of any quasi $p$-group. In \cite{16} it was shown that under certain conditions on the ramification filtration of the inertia group, its wild part can be made a little smaller.

The conjecture in general is wide open though some results supporting the conjecture have been obtained (see \cite{11}, \cite{12}, \cite{8}, \cite{15} and \cite{survey_paper} for more details).
 
For a quasi $p$-group $G$ and a subgroup $I$, we say that the pair $(G,I)$ is realizable if there exists a $G$-Galois cover $Y \to \PP^1$ \'{e}tale away from $\infty$ such that the inertia group at a point in $Y$ above $\infty$ is $I$. In this paper we first show that $(A_d,P)$ is realizable where $A_d$ is an Alternating group and $P$ is a $p$-subgroup containing a $p$-cycle (Corollary \ref{cor_A_d_containing_cycle}). Let $G_1$ and $G_2$ be perfect quasi $p$-groups, $I_1$ be a $p$-cyclic subgroup of $G_1$ and $I_2$ a cyclic $p$-subgroup of $G_2$. We show that if $(G_1,I_1)$ and $(G_2,I_2)$ are realizable then $(G_1\times G_2,I)$ is realizable for some cyclic $p$-subgroup $I$ of $G_1\times G_2$ such that its image in $G_1$ and $G_2$ are $I_1$ and $I_2$ respectively (Theorem \ref{thm_pdt_g}). As a consequence we obtain the following result.

\begin{corollary}
(Corollary \ref{cor_pdt_less_than_2p}) The wild part of the inertia conjecture is true for an $n$-fold product $A_{d_1} \times \cdots \times A_{d_n}$ where $p \leq d_i <2p$ for all $1 \leq i \leq n$.
\end{corollary}

In fact we show that if $(A_{rp},\langle \tau \rangle)$ and $(A_{rp+1},\langle \tau\rangle)$ are realizable for all $r\ge 2$ where $\tau$ is the product of $r$ disjoint $p$-cycles, then the wild part of the inertia conjecture holds for any finite product of Alternating groups.

The strategy of the proof is as follows. In \cite[Section 21]{7} Abhyankar introduced covers using some explicit equations and determined the Galois groups to be $A_d$, but the structure of the inertia group of such covers for $d \geq 2p$ has not been studied. For Alternating groups $A_d$, we compute these inertia groups and use Abhyankar's Lemma to prove that $(A_d,I)$ is realizable if $I$ is generated by a $p$-cycle. For Theorem \ref{thm_pdt_g} we use Harbater's formal patching techniques to construct covers with the inertia group of the form $(\Z/p)^2$ with a desirable ramification filtration so that one could apply \cite[Theorem 3.7]{16} and reduce the inertia group.

In Section \cref{sec_prilim} we introduce basic definitions and results regarding inertia groups and the Galois covers of smooth projective curves. Section \cref{sec_patching} contains some useful formal patching results needed to prove the main results. In Section \cref{sec_alt} we consider Alternating group covers. In Section \cref{sec_altp} we  treat the product of Alternating groups. Finally in Section \cref{sec_neweq}, for $d$, a multiple of $p$, we provide an explicit affine equation and show that its Galois closure is an $A_d$-Galois cover of $\mathbb{P}^1$ \'{e}tale away from $\infty$. The inertia group and related invariants are also calculated for this cover. We also show that this $A_d$-cover has the minimum possible upper jump (in the sense of \cite{9}).

\subsection*{Acknowledgements} To appear after the review process.

\section{Preliminaries}\label{sec_prilim}
Let $k$ be an algebraically closed field of characteristic $p>0$. A cover of $k$-curves means a finite generically separable morphism $\phi \colon Y\rightarrow X$ of smooth connected $k$-curves which is \'{e}tale outside a finite set of closed points in $X$. The automorphism group $\text{Aut}(Y/X)$ of a cover $\phi \colon Y \rightarrow X$ is the group of $k$-automorphisms $\sigma$ of $Y$ such that $\phi \circ \sigma =\phi$. For a finite group $G$, a $G$-Galois cover is a cover $Y\rightarrow X$ of curves together with an inclusion $\rho \colon G \hookrightarrow \text{Aut}(Y/X)$ such that $G$ acts simply transitively on each generic geometric fibre. Since $X$ is irreducible, the inclusion $\rho$ is necessarily an isomorphism.

\subsection{Group theory}
The following group theoretic results are useful to decide which subgroups of $A_d$ or of a product of Alternating groups are potentially inertia subgroups. A special case ($r=1$ and $d < 2 p$) of the next result can be found in \cite[Lemma 4.13, Lemma 4.14]{8}.

\begin{proposition}\label{prop2.1}
Let $p$ be a prime, $d \geq p$. Let $\tau$ be an element of order $p$ in the Symmetric group $S_d$. Let $\tau = \Pi_{i=1}^r \tau_i$ be the disjoint cycle decomposition of $\tau$ with $\tau_1,\cdots,\tau_r$ disjoint $p$-cycles and $r \cdot p \leq d$. For $\sigma \in S_d$ let $\Supp(\sigma)$ denote the support of $\sigma$. Then there exists an element $\theta \in \Sym(\Supp(\tau))\cap N_{S_d}(\langle \tau \rangle)$ of order $p-1$ such that conjugation by $\theta$ is a generator of $\Aut(\langle \tau \rangle)$. Moreover, let $H' \coloneqq \{\sigma\in S_d \colon \Supp(\sigma) \subset \Supp(\tau), \sigma \cdot \tau_i \cdot \sigma^{-1} = \tau_j \text{ for } 1 \leq i , j \leq r\}$. Then $N_{S_d}(\langle \tau \rangle)= \langle \theta, H'\rangle \times H$, where $H$ is the Symmetric group on the set $\{1,\cdots, d\} \setminus \Supp(\tau)$.

In particular, if $\beta\in N_{S_d}(\langle (1,\cdots, p) \rangle)$ has order prime to $p$ then $\beta=\theta^i \cdot \omega$ for some integer $0 \leq i \leq p-1$ and an element $\omega \in H = \Sym(\{p+1 , \cdots ,d\})$.
\end{proposition}

\begin{proof}
For a $p$-cycle $\tau'$ in $S_p$, the normalizer $N_{S_p}(\langle \tau' \rangle)$ is the affine general linear group $AGL(1,p)$ of order $p(p-1)$ and has an element of order $p-1$. Diagonally embedding $S_p$ in $\Sym(\Supp(\tau))$ we obtain an element $\theta$ of order $p-1$ in $\Sym(\Supp(\tau))$. Then $\theta$ normalizes $\langle \tau \rangle$ and the conjugation by $\theta$ has order $p-1$. Since $\tau$ is of order $p$, the full automorphism group of $\langle \tau \rangle$ is generated by $\theta$. So the natural homomorphism $N_{\Sym(\Supp(\tau))}(\langle \tau \rangle) \to \Aut(\langle \tau \rangle)$ is a surjection whose kernel is the centralizer of $\langle \tau \rangle$ in $\Sym(\Supp(\tau))$. Now observe that the centralizer of $\langle \tau \rangle$ in $\Sym(\Supp(\tau))$ is $H'$. So $N_{S_d}(\langle \tau\rangle) = \langle \theta , H' \rangle \times H$.
\end{proof}

\begin{notation}\label{not_pdt}
Let $u \geq 1$ and let $G=G_1 \times \cdots \times G_u$ where each $G_i$ is a finite group. For $g \in G$, let $g=(g^{(i)})_{1 \leq i \leq u}$, and set $S(g)=|\{1 \leq i \leq u| g^{(i)} \neq 1\}|$. Set $l(g)=|S(g)|$. For $\lambda \subseteq \{1, \cdots, u\}$, let $H_\lambda \coloneqq \Pi_{i \in \lambda} G_i$ and let $\pi_\lambda \colon G \twoheadrightarrow H_\lambda$ be the projection. 
\end{notation}

\begin{lemma}\label{lem_minimal_inertia_product}
Let $u \geq 1$ be an integer, $G=G_1 \times \cdots \times G_u$, where each $G_i$ is a simple non-abelian quasi $p$-group. Let $Q$ be a $p$-subgroup whose conjugates generate $G$. Then there are elements $g_1$, $\cdots$, $g_r$ in $Q$ for some $r \geq 1$ satisfying the following properties. Set $S_{\le j} \coloneqq \cup_{i=1}^j S(g_i)$, $1 \leq j \leq r$.
\hfill
\begin{enumerate}
\item $S_{\le r} = \{1, \cdots, u\}$, $l(g_1) \geq \cdots \geq l(g_r)$ and for all $1\leq i,j\leq r$, $S(g_i) \cap S(g_j) \neq \varnothing$.
\item For all $1 \leq i \leq r$, $H_{S_{\le i}}$ is generated by the conjugates of $\langle g_1, \cdots, g_i \rangle$.
\item For any subset $\{i_1,\cdots,i_t\} \subset \{1,\cdots, r\}$ and any integers $1 \leq a_{i_1} , \cdots , a_{i_t} \leq p-1$, $l(g_{i_1}^{a_{i_1}} \cdots g_{i_t}^{a_{i_t}}) \leq \text{max}_{1 \leq j \leq t} l(g_{i_j})$.
\item For each $1 \leq i \leq r$, $\ord(g_i^{(j)}) = p$ for some $1 \leq j \leq u$.
\end{enumerate}
\end{lemma}

\begin{proof}
Choose $g_1 \in Q$ such that $l(g_1) \geq l(g)$ for all $g \in Q$. If $l(g_1) = u$ then set $r=1$ and note that conditions (1) and (3) trivially are satisfied. Condition (2) holds because $G_j$'s are simple groups. Inductively define $g_i$ as follows. For $i\ge 1$, if $S_{\le i}$ is not the whole set then choose $g_{i+1}$ among $g\in Q$ with $S(g)\cap (S_{\le i})^c \ne \varnothing$ such that $l(g_{i+1})$ is maximal. The inclusion $S_{\le i}\subset S_{\le i+1}$ implies $l(g_i)\ge l(g_{i+1})$. Since $G_j$'s are all simple non-abelian groups, the conjugates of $\langle g_1,\cdots,g_i\rangle$ generate $H_{S_{\le i}}$. For $j\le i-1$, $l(g_ig_j) \leq l(g_j)$ and hence $S(g_i)\cap S(g_j) \ne \varnothing$. Condition (3) follows by the choice of $g_j$'s to maximize $l(g_j)$'s. 

Finally for condition (4) note that if $p^{k+1}$ is the least order of $g_i^{(j)}$ for various $j$ among nontrivial $g_i^{(j)}$ then replacing $g_i$ by $g_i^{p^k}$ will achieve the goal.
\end{proof}

\begin{remark}
In the above lemma if $1 \leq u \leq p$ then $r=1$, i.e. there exist $g\in Q$ such that $S(g)=\{1,\cdots, u \}$. Also note that we can take each $l(g_i)\geq p$.
\end{remark}

The following lemma follows from the fact that the abelianization of a quasi $p$-group is a $p$-group.

\begin{lemma}\label{lem_grp_perfect}
Let $G$ be a finite quasi $p$-group. Then the following are equivalent.
\begin{enumerate}
\item $G$ is perfect (i.e. $G$ equals its commutator subgroup).
\item There is no nontrivial homomorphism from $G$ to $\mathbb{Z}/p$.
\item There is no nontrivial homomorphism from $G$ to any $p$-group.
\end{enumerate}
\end{lemma}

\subsection{Ramification}
Let $X$ be a smooth projective curve over $k$, with function field $K$. Let $F/K$ be a finite separable extension. Let $\psi \colon Y \to X$ be the normalization of $X$ in $F$. Let $\phi \colon \widetilde{Y}\rightarrow X$ be the Galois closure of $\psi$ with Galois group $G$. It is well-known (\cite[Lemma 4.2]{8}) that the branch loci of $\phi$ and $\psi$ are the same, which we denote by $B$. For details on the ramification theory for local extensions see \cite[Chapter IV]{Serre_loc}, from which we recall a few facts.

Let $x \in B$, $y \in \phi^{-1}(x)$, and let $I$ be the inertia group at $y$. Since $k$ is algebraically closed and $y$ is a closed point of $Y$, the inertia group is the same as the decomposition group at $y$. Note that $I$ is of the form $P \rtimes \mu_m$ where $P$ is a $p$-group and $(m,p)=1$ (\cite[Chapter IV, Corollary 4]{Serre_loc}). Also recall that the lower ramification groups $\{I_i\}_{i\geq 0}$ constitute a finite decreasing filtration of $I$ at $y$ which is in an explicit bijection with the upper indexed ramification filtration of $I$. If $|I|=pm$, $p\nmid m$, there is a unique lower jump $h\geq 1$ which is called the \textit{conductor}. The corresponding \textit{upper jump} is denoted by $\sigma\coloneqq h/m$.

We fix the following setup for the rest of the section. Let $p>2$ be a prime, $d \geq p+2$ an integer and $\psi \colon Y\rightarrow \mathbb{P}^1$ a degree-$d$ smooth connected cover \'{e}tale away from $\infty$. Let $\phi \colon \widetilde{Y}\rightarrow \mathbb{P}^1$ be its Galois closure with Galois group $G$. Note that $G$ is a transitive quasi $p$-subgroup of $S_d$. Assume that the inertia group at a point in $\widetilde{Y}$ lying over $\infty$ is $I = P \rtimes \mu_m$, where $P\cong \mathbb{Z}/p \mathbb{Z}$. Let $\tau$ be a (fixed) generator of $P$, and $\beta\in \mu_m$ is of order $m$. For the rest of this section we also assume that $\tau$ is the $p$-cycle $(1,\cdots,p)$. Let $h$ be the conductor. Also $I = N_{I}(\langle \tau \rangle)\subseteq N_{G}(\langle \tau \rangle)=N_{S_d}(\langle \tau \rangle)\cap G$. By Proposition \ref{prop2.1}, $\beta=\theta^i\omega$ for some $\theta$, $\omega\in N_{S_d}(\langle \tau \rangle)$ and $0\le i \le p-1$. Furthermore $\theta\in \Sym(\Supp(\tau))=S_p$ is of order $p-1$ and $\omega\in \Sym(\Supp(\tau)^c)$. 

Consider the group homomorphism $g \colon \langle \beta \rangle \rightarrow \text{Aut }(\langle \tau \rangle)$, $g(\beta)\colon \tau \mapsto \beta \tau \beta^{-1}$. Let $m'$ be the order of $\text{ker}(g)=\{\beta^j | \beta^j \text{ commutes with } \tau\}$. This is the prime-to-$p$ part of the center of $G_0$, called \textit{the central part of the tame ramification}. Put $m''\coloneqq \frac{m}{m'}$. Then $\text{ker}(g)=\langle\beta^{m''}\rangle$ is the subgroup acting trivially on $\langle \tau \rangle$. Also observe that $\text{Im}(g)$ has order $m''$ in $\text{Aut}(\langle \tau \rangle)$, and $m''$ is called \textit{the faithful part of the tame ramification}. Let $h$ be the conductor. The following two lemmas are easy generalizations of results in \cite[Section 4.6, Proposition 4.16 and 4.17]{survey_paper}.

\begin{lemma}\label{lem2.1}
Under the above notation $m'=\text{g.c.d.}(h,m)$ and $\ord(\theta^i)=m''$.
\end{lemma}

\begin{proof}
Set $\gamma\coloneqq (h,m)$. Applying \cite[Chapter IV, Proposition 9]{Serre_loc} to the conjugation of $\tau$ by $\beta^{\frac{m}{\gamma}}$ and $\beta^{m''}$, we get that $\beta^{\frac{m}{\gamma}} \tau \beta^{-\frac{m}{\gamma}} = \tau$ and $\beta^{m''h}=1$.  So $m''$ divides $\frac{m}{\gamma}=\frac{m'm''}{\gamma}$ and $m'm''$ divides $m''h$. So $\gamma$ divides $m'$ and $m'$ divides $(h,m)=\gamma$. Thus $\gamma=m'$.

Note that the conjugation by $\theta$ has order $p-1$ (Proposition \ref{prop2.1}). Since $\theta$ also has order $p-1$, $\theta^j$ commutes with $\tau$ if only if $\theta^j$ is the identity. Note that $\beta=\theta^i \cdot \omega$ and $\omega$ commutes with $\tau$. So the conjugation by $\beta$ is the same as the conjugation by $\theta^i$. Hence $(\theta^i)^k$ does not commute with $\tau$ for $0<k<m''$, and $(\theta^i)^{m''}$ commutes with $\tau$. So $m''$ is the order of $\theta^i$.
\end{proof}

Since $\tau$ and $\theta$ are elements in $S_p$, the action of $I$ on $\{p+1,\cdots ,d\}$ is the same as that of $\beta$ and hence is the same as the action of $\omega$. Suppose that the action of $\beta$ on $\{p + 1, \cdots, d\}$ breaks into $t$ disjoint cycles of lengths $n_1$, $\cdots$, $n_t$ (equivalently, the disjoint cycle decomposition of $\omega \in \text{Sym} \{p+1,\cdots,d\}$ consists of $t$ disjoint cycles of length $n_1$, $\cdots$, $n_t$). The following result was proved in \cite{survey_paper} with the assumption that the order of the group $G$ is strictly divisible by $p$. We note that the same proof works without this assumption on $|G|$ in view of Proposition \ref{prop2.1}.

\begin{lemma}\label{lem2.2}
(cf. \cite[Proposition 4.16]{survey_paper}) Under the above notation the fibre of $\psi \colon Y \rightarrow \mathbb{P}^1$ above $\infty$ consists of $t+1$ points with ramification indices of order $p$, $n_1$, $\cdots$, $n_t$.
\end{lemma}

Since the action of $\beta$ on $\{p+1,\cdots , d\}$ is the same as that of $\omega \in H$, $\omega$ is a product of $t$ disjoint cycles of lengths $n_1$, $\cdots$, $n_t$. Set $L \coloneqq \text{ L.c.m. } (n_1,\cdots,n_t)$. Also since $\beta$ has order prime-to-$p$, for each $i$, $(p,n_i)=1$. Let $g(Y)$ be the genus of the curve $Y$. Using the Riemann-Hurwitz formula and Hilbert's different formula (\cite[IV, Proposition 4]{Serre_loc}) to the Galois covers $\psi \colon Y \to \PP^1$ and to $\widetilde{Y} \to Y$ (a technique used in \cite[Proposition 1.3]{9} when $p$ strictly divides the order of $G$), we obtain
\begin{equation}\label{eq1}
\sigma=\frac{h}{m}=\frac{2g(Y)+d+t-1}{p-1}.
\end{equation}

By Lemma \ref{lem2.1}, $m''$ is the smallest positive integer such that $m''\cdot \sigma$ is again an integer. So we have
\begin{equation}\label{eq2}
m''=\frac{p-1}{(p-1,2g(Y)+d+t-1)}.
\end{equation}

Since order of $\theta^i$ is $m''$ by Lemma \ref{lem2.1}, and $m' m''=m=\text{ l.c.m. }(\text{ord }(\theta^i), \text{ord }(\omega))$, we deduce that
\begin{equation}\label{eq3}
m'=\frac{L}{(L,m'')}.
\end{equation}

Let us summarize the above observations in the following Theorem.

\begin{theorem}\label{thm_grp}
Let $p$ be an odd prime, $d \geq p+2$, and $G$ be a transitive subgroup of $S_d$. Let $\tau \in G$ be a $p$-cycle, $P=\langle \tau \rangle$. Let $\phi \colon \widetilde{Y} \rightarrow \mathbb{P}^1$ be a $G$-Galois cover \'{e}tale away from $\infty$, with inertia group $I=P \rtimes \mu_m$ above $\infty$, $(p,m)=1$. Let $Y$ denote the quotient $\widetilde{Y}/(G\cap S_{d-1})$, and $\psi \colon Y \rightarrow \mathbb{P}^1$ be the degree-$d$ cover. If $\beta$ is a generator of the cyclic group $\mu_m$, then $\beta$ is of the form $\theta^i \cdot \omega$, for some $1\le i\le p-1$, where $\theta$ and $\omega$ are as in Proposition \ref{prop2.1}. There are prime-to-$p$ integers $n_1$, $\cdots$, $n_t$, $t \geq 1$ ($t \geq 2$ when $p|d$), such that $\omega$ is a disjoint product of $t$-many cycles of length $n_i$'s, and there are exactly $t+1$ points in the fibre of $\psi$ above $\infty$ with ramification indices $p$, $n_1$, $\cdots$, $n_t$. Moreover, the the upper jump $\sigma= h/m$, and the integers $m''$, $m'$ are described by the Equations \ref{eq1}, \ref{eq2} and \ref{eq3}, respectively.
\end{theorem}

\section{Formal patching}\label{sec_patching}
We fix the following notation.

\begin{notation}\label{not+}
Let $R=k[[t]]$, $\widetilde{K}=k((t))$, $U=\text{Spec}(k[[x^{-1}]])$. Let $b$ be the closed point of $U$. For any $k$-algebra $A$ with $K=QF(A)$ and any $k$-scheme $W$, let $W_A \coloneqq W \times_k A$, $W_K \coloneqq W_A \times_A K$, and for any closed point $w \in W$, $w_A \coloneqq w \times_k A$, $w_K \coloneqq w_A \times_A K$. For an $R$-scheme $V$, let $V^0$ denote the closed fiber of $V\to \spec(R)$. 
\end{notation}

Given a Galois cover $f:Y\to X$ with inertia group $I$ at a point $y\in Y$ and $f(y)=x$, we deform the local $I$-Galois cover at $x$ to obtain a new cover with desired local properties. We use a formal patching argument (cf. \cite[Theorem 2]{2} or \cite[Theorem 3.6]{AC_embed_Harbater}) to first get a cover over $X_R$ and then use a Lefschetz type principle to obtain Galois covers of smooth projective $k$-curves.

\begin{lemma}\label{lem_formal_patching}
(Formal Patching) Let $G$ be a finite group and $X$ be an irreducible smooth projective $k$-curve. Let $I_1 \subset I$ and $H$ be subgroups of $G$ such that $G = \langle H,I \rangle$. Assume that there is an irreducible $H$-Galois cover $\phi \colon Y \rightarrow X$ of smooth projective curves branched only at a point $x \in X$, with inertia group $I_1$ at a point $y$ in $Y$ above $x$. Let $U_{X,x} = \spec(\widehat{\mathcal{O}}_{X,x})$, $K_{X,x} = QF(\widehat{\mathcal{O}}_{X,x})$ and $L = QF(\widehat{\mathcal{O}}_{Y,y})$. Let $b_x$ denote the closed point of $U_{X,x}$. Let $Z \rightarrow U_{X,x,R}$ be an irreducible $I$-Galois cover of integral $R$-schemes totally ramified over $b_{x,R}$ such that the normalization of the pullback of $Z^0 \rightarrow U_{X,x}$ to $\spec(K_{X,x})$ is isomorphic to $\Ind_{I_1}^{I} \spec(L)$ as the $I$-Galois covers of $\spec(K_{X,x})$. Then there is a normal $G$-Galois cover $V \rightarrow X_R$ of irreducible $R$-curves such that the following hold.
\begin{enumerate}
\item $V \rightarrow X_R$ is \'{e}tale away from $x_R$ with inertia group $I$ above $x_R$;
\item $V \times_{X_R} U_{X,x,R} \cong \Ind_I^G Z$ and if $X'=X \setminus \{x\}$ then $V \times_{X_R} X' _R \cong \Ind_H^G Y_R'$ where $Y'=\phi^{-1}(X')$;
\item the closed fibre $V^0$ is connected and the normalization of the pullback of $V^0\rightarrow X$ to $\spec(K_{X,x})$ is isomorphic to $\Ind_{I_1}^G \spec(L)$ as $G$-Galois cover of $\spec(K_{X,x})$.
\end{enumerate}
\end{lemma}

\begin{proof}
Consider the trivial deformation $Y'_R \rightarrow X'_R$ of the $H$-Galois cover $Y' \rightarrow X'$. Taking disjoint union of $[G:H]$-copies of $Y'_R$, we obtain a (disconnected) normal $G$-Galois cover $W_1 \coloneqq \Ind_{H}^G Y'_R \rightarrow X'_R$, in which the stabilizers of the components are the conjugates of $H$ in $G$. Now taking disjoint union of $[G:I]$-copies of $Z$, we obtain a (disconnected) normal $G$-Galois branched cover $W_2 \coloneqq \Ind_{I}^G Z \rightarrow U_{X,x,R}$, in which the stabilizers of the components are the conjugates of $I$ in $G$. The fibre of $W_1^0$ over $\spec(K_{X,x})$ is given by $\Ind_H^G \Ind_{I_1}^H \spec(L)$ and that of $W_2^0$ over $\spec(K_{X,x})$ is given by $\Ind_I^G  \Ind_{I_1}^{I} \spec(L)$. Since both of these covers are indexed by left cosets of $I_1$ in $G$, we can choose an isomorphism between these fibres that is compatible with the indexing, and hence with the $G$-action. Since the pullbacks of $W_1 \rightarrow X'_R$ and $W_2 \rightarrow U_{X,x,R}$ to $\spec(K_{X,x,R})$ are \'{e}tale, by \cite[I, Corollary 6.2]{10}, they are trivial deformations of the pullbacks of $W_1^0 \rightarrow X'$ and $W_2^0 \rightarrow U_{X,x}$ to $\spec(K_{X,x})$. So the above isomorphism lifts uniquely to a $G$-isomorphism $W_1 \times_{X_R'} K_{X,x,R} \cong W_2 \times_{U_{X,x,R}} K_{X,x,R}$ over $\spec(K_{X,x,R})$. By \cite[Proposition 4(b)]{2}, there exists a unique normal $G$-Galois cover $V \rightarrow X_R$ such that $V\times_{X_R} X'_R \cong W_1$ and $V\times_{X_R} U_{X,x,R}\cong W_2$ as covers of $X'_R$ and $U_{X,x,R}$ respectively. Since $W_2$ has branch locus $x_R$ and $W_1$ is \'{e}tale, $V \rightarrow X_R$ has branch locus $\{x_R\}$. Also since $Z \rightarrow U_{X,x,R}$ is totally ramified above $b_{x,R}$, we have $\Gal(V \times_{X_R} (K_{X,x} \otimes_k \widetilde{K}))/(K_{X,x} \otimes_k \widetilde{K})) = \Gal(Z \times_{U_{X,x,R}} (K_{X,x} \otimes_k \widetilde{K}))/(K_{X,x} \otimes_k \widetilde{K})) = I$. So the inertia group above $x_R$ is $I$.  Finally, since the the stabilizers of the identity components of $W_1 \rightarrow X'_R$ and of $W_2 \rightarrow U_{X,x,R}$ are $H$ and $I$ respectively, the stabilizer of the identity component of $V \rightarrow X_R$ is $\langle H, I \rangle = G$. So $V$ is irreducible.
\end{proof}

\begin{lemma}\label{lem_dense_realization}
Let $X$ be an irreducible smooth projective $k$-curve, and $x \in X$ be a closed point. Let $U_{X,x} = \spec(\widehat{\mathcal{O}}_{X,x})$ and $K_{X,x} = QF(\widehat{\mathcal{O}}_{X,x})$. Let $b_x$ denote the closed point of $U_{X,x}$. For $1 \leq i \leq r$, let $G_i$ be a finite group and let $I_i$ be a subgroup of $G_i$. Assume that for each $1 \leq i \leq r$, there is an $I_i$-Galois cover $\phi_i \colon \widetilde{S}_i \rightarrow U_{X,x,k[t]}$ of integral schemes with branch locus $b_{x,k[t]}$ over which it is totally ramified. For those closed points $\beta \in \A^1_t$ where the fiber $\widetilde S_{i,\beta}$ of $\widetilde S_i\to \A^1_t$ is integral, let $M^{(i)}_{\beta}/K_{X,x}$ denote the field extension corresponding to the cover $\widetilde S_{i,\beta}\to U_{X,x}$ at $t=\beta$. For $1\le i \le r$, let $Z_i \coloneqq \widetilde{S}_i \times_{U_{X,x,k[t]}} U_{X,x,R}$, and $V_i \rightarrow X_R$ be a normal $G_i$-Galois cover of irreducible $R$-curves \'{e}tale away from $x_R$ with inertia group $I_i$ above $x_R$ such that $V_i \times_{X_R} U_{X,x,R} \cong \Ind_{I_i}^{G_i} Z_i$. Then there is an open dense subset $\mathcal{V}$ of $\mathbb{A}^1_t$ such that for all closed points $(t=\beta)$ in $\mathcal{V}$ the following holds. For each $1\leq i \leq r$, there is a $G_i$-Galois cover $W_i \rightarrow X$ branched only at $x$ with inertia group $I_i$ at a point in $W_i$ above $x$ such that the local $I_i$-Galois extension corresponding to the formal neighbourhood at that point is $M^{(i)}_\beta/K_{X,x}$.
\end{lemma}

\begin{proof}
 Let $1 \leq i \leq r$. Consider the $G_i$-Galois covers $f^{(i)} \colon V_i \times_{X_R} U_{X,x,R} \rightarrow U_{X,x,R}$ and $g^{(i)} \colon \Ind_{I_i}^{G_i} Z_i \rightarrow U_{X,x,R}$. Since $f^{(i)}$ and $g^{(i)}$ are finite morphisms and are $G$-Galois covers, the isomorphism $V_i \times_{X_R} U_{X,x,R} \cong \Ind_{I_i}^{G_i} Z_i$ is equivalent to a $G_i$-equivariant isomorphism of coherent sheaves $f^{(i)}_*(\mathcal{O}_{V_i \times_{X_R} U_{X,x,R}})$ and $g^{(i)}_*(\mathcal{O}_{\Ind_{I_i}^{G_i} Z_i})$ over $U_{X,x,R}$, and hence it is defined locally by matrices involving only finitely many functions over $U_{X,x,R}$. So there exists a finite type $k[t]$-algebra $A \subset R$ having smooth connected spectrum $E=\spec(A)$ and for each $1 \leq i \leq r$, an irreducible $G_i$-Galois cover $\pi_i \colon F_i \to X_A$ branched only over $x_E$ with inertia group $I_i$ above $x_E$, together with an isomorphism $F_i \times_{X_A} U_{X,x,A} \cong \Ind_{I_i}^{G_i} \widetilde{S}_i \times_{U_{X,x,k[t]}} U_{X,x,A}$ such that $F_i \times_A R \cong V_i$, and the fibre over each closed point of $E$ is irreducible and non-empty. So for each point $e \in E$ and for each $1\leq i \leq r$, $\widetilde{F_{i,e}} \rightarrow X \times_k \{e\} \cong X$ is a $G_i$-Galois cover \'{e}tale away from $x$ with inertia group $I_i$ above $x$, where $\widetilde{F_{i,e}}$ is the normalization of the fibre $F_{i,e} = \pi_i^{-1}(X \times_k \{e\})$. Since the ring map $k[t] \rightarrow A$ is also injective, the finite type map $E \rightarrow \mathbb{A}^1_t$ is flat and dominant. So the image of $E$ in $\mathbb{A}^1_t$ is an open dense set which is our $\mathcal{V}$. Now for every point $\beta \in \mathcal{V}$ with preimage $e_\beta \in E$ and for each $1\leq i \leq r$, the corresponding fibre $W_i \coloneqq \widetilde{F}_{i,e_\beta} \rightarrow X$ is a $G_i$-Galois cover branched only at $x$, and the $I_i$-Galois extension corresponding to the formal neighborhood of a point in $W_i$ lying above $x$ is $M^{(i)}_\beta/K_{X,x}$.
\end{proof}

\begin{remark}
Note that the above two Lemmas \ref{lem_formal_patching}, \ref{lem_dense_realization} can be easily generalized to the case with branch locus containing more than one point.
\end{remark}

The following application of patching result by Raynaud (\cite[Theorem 2.2.3]{3}) will also be used later.

\begin{theorem}(Raynaud)\label{thm_Raynaud}
Let $G$ be a finite group and let $Q$ be a $p$-subgroup of $G$. Let $J$ be a finite set. For each $i \in J$, let $G_i$ be a subgroup of $G$ and $Q_i$ be a subgroup of $G_i \cap Q$. Assume that $G = \langle \{G_i\}_{i \in J}, Q \rangle$. Suppose that for each $i \in J$, there is a connected $G_i$-Galois \'{e}tale cover of the affine line with inertia group $Q_i$ above $\infty$. Then there is a connected $G$-Galois \'{e}tale cover of the affine line with inertia group $Q$ above $\infty$.
\end{theorem}

\section{Wild Part of the Inertia Conjecture for Alternating groups}\label{sec_alt} 
Let $p$ be an odd prime. In this section we prove the wild part of the inertia conjecture for Alternating group $A_d$ when the inertia group contains a $p$-cycle. Note that $A_d$ is a quasi $p$-group for $d\ge p$ and when $p \leq d < 2p$ the wild part of the conjecture is immediate. So we may assume that $d \geq 2p$. We begin with the following lemma on the Galois groups of local fields. For a field $L$ and a separable polynomial $f(Z)\in L[Z] $, denote the Galois group of the splitting field of $f(Z)$ over $L$ by $\Gal(f,L)$.

\begin{lemma}\label{lem3.3}
Let $h(Z)\in k((x))[Z]$ be an irreducible polynomial of degree $m$ where $(m,p) = 1$. Then $\Gal(h, k((x))) = \mathbb{Z}/m\mathbb{Z}$.
\end{lemma}

\begin{proof}
Let $L = k((x))[Z]/(h(Z))$ and $L_s$ the splitting field of $h$ over $k((x))$. Let $G =\Gal(L_s/k((x)))$ and $G' = \Gal(L_s/L)$. Note that $G = P \rtimes \mu_m$ where $P$ is a $p$-group and $p \nmid m$. Since $[L:k((x))] = m$, $G'$ is an index $m$ subgroup of $G$. But $(m, p) = 1$ implies $P \leq G'$. Hence $L \subset L_s^P$. But $L_s^P/k((x))$ is a Galois extension with cyclic Galois group. Hence $L/k((x))$ is also a cyclic Galois extension. So $L$ is the splitting field of $h$ and $\Gal(L/k((x))) = \mu_m$.
\end{proof}

\begin{definition}\label{def3.1}
Let $G$ be a quasi $p$-group, and $I$ be a subgroup of $G$. Recall that the pair $(G,I)$ is realizable if there exists a $G$-Galois cover $Y \rightarrow \mathbb{P}^1$ branched only at $\infty$, with inertia group $I$ at a point above $\infty$. An $I$-Galois extension $L/k((x^{-1}))$ is said to be realized by the pair $(G,I)$ if there is a $G$-Galois cover of $\mathbb{P}^1$ branched only at $\infty$ with inertia group $I$ above $\infty$ so that the corresponding local $I$-Galois extension is isomorphic to $L/k((x^{-1}))$.
\end{definition}

By Harbater's result (\cite[Theorem 2]{2}), it is enough to show that $(A_d,\langle \tau\rangle)$ is realizable when $\tau$ is a $p$-cycle. For this we use an equation of an $A_d$-cover given by Abhyankar.

\begin{proposition}\label{prop_3.3B}
Let $p$ be an odd prime, $d=p + s \geq p+2$, $p \nmid s$, $(p,s) \neq (7,2)$. Consider the degree $d$ cover $\psi \colon Y \rightarrow \mathbb{P}^1$ given by the affine Equation
\begin{equation}\label{eq5}
x=\frac{1+y^d}{y^s}.
\end{equation}
Let $\phi \colon \widetilde{Y} \rightarrow \mathbb{P}^1$ be its Galois closure. Then $\phi$ is an $A_d$-Galois cover of $\mathbb{P}^1$ branched only at $\infty$, with inertia group $I=\langle \tau \rangle \rtimes \langle \beta \rangle$ above $\infty$ where $\tau$ is a $p$-cycle, and $\beta$ has order $\text{l.c.m.}(\frac{p-1}{(p-1,s+1)},s)$. Furthermore, it has upper jump $\frac{d}{p-1}$.
\end{proposition}

\begin{proof}
By \cite[Section 11]{7}, $G=A_d$ and the cover $\phi$ is branched only at $\infty$. Set $x_0=x^{-1}$ and consider the local equation of $Y$ near a preimage of $x=\infty$. Then $I$ is the Galois group of the splitting field of $f(y)= x_0 \cdot y^d - y^s + x_0$ over $k((x_0))$. By Hensel's Lemma, $f(y)=g(y) \cdot h(y)$ in $k[[x_0]][y]$ with $g(y) \equiv -1 (\text{ mod } x_0)$ of degree $p$ and $h(y) \equiv y^{s} (\text{ mod } x_0)$ of degree $s$. So we have
$$\text{Gal}(f(y),k((x_0))) \leq \text{Gal} (g(y), k((x_0))) \times \text{Gal} (h(y), k((x_0))).$$
Let $h(y)= b_0 +\cdots + b_s y^{s}$ for some $b_i \in k[[x_0]]$. Then $\text{ord}(b_i) \geq 1$ for $1 \leq i \leq s-1$, and $\text{ord}(b_{s})=0$. Since $f = g \cdot h$, $b_0$ divides $x_0$, and so $\text{ord} (b_0)=1$. So $h(y)$ is Eisenstein, and hence it is an irreducible polynomial in $k((x_0))[y]$. Since $h$ is irreducible of prime-to-$p$ degree $s$, by Lemma \ref{lem3.3}, $\text{Gal}(h ,k((x_0))) =\mu_{s}$. Also $\text{Gal}(g,k((x_0)))\leq \text{Gal}(f, k((x_0))) \leq \text{Gal}(f, k(x_0)) = A_d$. Moreover, $p$ divides the order of the inertia group $I$ and $g$ has degree $p$ and so $\text{Gal}(g, k((x_0)))$ is a transitive subgroup of $S_p$. Hence $\text{Gal}(g, k((x_0))) \leq A_p$. So the Sylow $p$-subgroup $P$ of $\text{Gal}(g, k((x_0)))$ is generated by a $p$-cycle $\tau$. Since $I \subset N_{A_d}(\langle \tau \rangle)=N_{S_d}(\langle \tau \rangle) \cap A_d$, by Proposition \ref{prop2.1}, the inertia group $I$ is of the form $\langle \tau \rangle \rtimes \langle \beta \rangle$, with $\beta = \theta^i \cdot \omega$ for an element $\theta$ of order $(p-1)$, an integer $i$ and $\omega \in H = \text{Sym}(\{p+1 , \cdots ,d\})$. Since $x$ is a rational function of $y$, $Y \cong \mathbb{P}^1$. Note that there are exactly two points in the fibre $\psi^{-1}(\infty)$ with ramification indices $p$ and $s$. So by Lemma \ref{lem2.2}, $\omega$ is an $s$-cycle. Let $h$ be the conductor and let $\beta$ has order $m$. Then by Equation \ref{eq1}, the upper jump $\frac{h}{m}=\frac{d}{p-1}$. Now Equation \ref{eq2} implies $\text{ord}(\theta^i) = m'' = \frac{p-1}{(p-1,d)} = \frac{p-1}{(p-1,s+1)}$. So we have $m = \text{l.c.m.}(\text{ord}(\theta^i),\text{ord}(\omega)) = \text{l.c.m.}(\frac{p-1}{(p-1,s+1)},s)$.
\end{proof}

Observe that the above Proposition \ref{prop_3.3B} was proved in \cite[Theorem 4.9]{8} using a different method when $s<p$.

\begin{corollary}\label{thm3.3B}
Let $d=p+s$, $p \nmid s$. Then for any $p$-cycle $\tau$, the pair $(A_d,\langle \tau\rangle)$ is realizable.
\end{corollary}

\begin{proof}
 By Proposition \ref{prop_3.3B}, the $A_d$-Galois cover of $\mathbb{P}^1$ obtained from Equation \eqref{eq5} is branched only at $\infty$ with inertia group $I$ whose Sylow $p$-subgroup is generated by a $p$-cycle. Now the result follows by applying Abhyankar's Lemma to this cover.
\end{proof}

\begin{corollary}\label{thm3.3A}
Let $d = ap$ for some integer $a \geq 1$. Then for any $p$-cycle $\tau$, the pair $(A_d,\langle \tau\rangle)$ is realizable.
\end{corollary}

\begin{proof}
We may assume that $a\geq 2$. Consider the $A_{d+1}$-Galois cover $\phi \colon \widetilde{Y} \rightarrow  \mathbb{P}^1$ branched only at $\infty$, which is the Galois closure of the degree-$d$ cover $\psi \colon Y \rightarrow \mathbb{P}^1$ given by the affine Equation
\begin{equation}
x=\frac{1+y^{d+1}}{y^{d-p+1}}.
\end{equation}
By Proposition \ref{prop_3.3B}, the inertia group $I$ above $\infty$ is given by $\langle \tau \rangle \rtimes \langle \beta \rangle \cong \mathbb{Z}/p \rtimes \mu_m$ for a $p$-cycle $\tau$ and for the prime-to-$p$ integer $m = \text{l.c.m.}(\frac{p-1}{(p-1,d+1)},d-p+1)$. Consider the $A_d$-Galois cover $\epsilon \colon \widetilde{Y} \rightarrow Y \cong \mathbb{P}^1$. By Lemma \ref{lem2.2}, this cover is branched at two points $(y=0)$ and $(y=\infty)$ with respective inertia groups of order $p\frac{m}{d-p+1}$ and $m$ above them. Let $\eta_m \colon \PP^1 \to \mathbb{P}^1$ be the $m$-cyclic Kummer cover by sending $z^m$ to $y$. Since $A_d$ is a quasi $p$-group, $\epsilon$ is an $A_d$-cover and $\eta_m$ is prime-to-$p$ Galois cover, they are linearly disjoint. By the refined Abhyankar's lemma (\cite[Lemma 4.1]{8}), the normalized pullback $Z \rightarrow \mathbb{P}^1$ of $\epsilon$ along $\eta_m$ is an $A_d$-Galois cover branched only at $\infty$, and the inertia group above $\infty$ is generated by the $p$-cycle $\tau$.
\end{proof}

\begin{remark}\label{rmk_a_d_ conductor}
From Proposition \ref{prop_3.3B}, it follows that the conductor $h$ for the cover $\phi$ is given by $h=\frac{d+1}{p-1} \times \text{l.c.m.}(\frac{p-1}{(p-1,d+1)},d-p+1)$. By \cite[IV, Proposition 2]{Serre_loc}, it follows that the cover $\epsilon$ also has conductor $h$. Applying \cite[Lemma 4.1]{20} shows that the $A_d$-Galois cover $Z \rightarrow \mathbb{P}^1$ also has the same conductor $h$. In Section \ref{sec_neweq}, we introduce a different cover which has the minimal possible upper jump $\frac{d+1}{p-1}$ and offers an alternate proof (cf. Proposition \ref{prop3.1}) for Corollary \ref{thm3.3A} when $a \geq 3$.
\end{remark}

The following result follows immediately from Corollary \ref{thm3.3B} and Corollary \ref{thm3.3A} together with Harbater's result \cite[Theorem 2]{2}.

\begin{corollary}\label{cor_A_d_containing_cycle}
Let $P$ be a $p$-subgroup of $A_d$ containing a $p$-cycle. Then the pair $(A_d,P)$ is realizable.
\end{corollary}

\begin{remark}\label{rmk_general_IC_for_A_d}
Note that to prove the wild part of the inertia conjecture for $A_d$ one needs to show that $(A_d, \langle \tau \rangle)$ is realizable where $\tau$ is a product of $r$-many disjoint $p$-cycles for all $r$ with $1 \leq r \leq \floor*{\frac{d}{p}}$.
\end{remark}

The following result shows that to prove the wild part of the inertia conjecture for $A_d$, $d \geq p$, it is enough to prove the conjecture for the cases $d \equiv 0 \text{ or } 1 (\text{ mod } p)$.

\begin{proposition}\label{prop_reduction_to_one_mod_p}
Let $r \geq 2$ be an integer. Assume that the pair $(A_{rp+1}, \langle \tau \rangle)$ is realizable where $\tau$ is the product of $r$ disjoint $p$-cycles in $A_{rp+1}$. Then for any $d\geq rp+1$, the pair $(A_d,\langle \tau \rangle)$ is realizable.
\end{proposition}

\begin{proof}
This is an immediate consequence of Raynaud's result (Theorem \ref{thm_Raynaud}) by taking $G_i=\Alt(\Supp(\tau) \cup \{i\})$ and $Q = Q_i=\langle \tau \rangle$ for $i \in \{1,\cdots,d\} \setminus \Supp(\tau)$.
\end{proof}

\section{Wild Part of the Inertia Conjecture for product of Alternating groups}\label{sec_altp}
The first two results of this section use formal patching results from section \cref{sec_patching} to construct $G_1\times G_2$ covers of $\PP^1$ from the given $G_1$ and $G_2$ covers such that the inertia group over $\infty$ is smaller than the one obtained from the fiber product of the two covers. This will be used to construct a product of Alternating group covers with a certain cyclic $p$-group as the inertia group. In view of Lemma \ref{lem_minimal_inertia_product} and \cite[Theorem 2]{2}, this is exactly what we need.

\begin{lemma}\label{lem_prep}
Let $G_1$, $G_2$ be two quasi $p$-groups, $P_1$ and $P_2$ be $p$-subgroups of $G_1$ and $G_2$ respectively. Assume that the pairs $(G_1,P_1)$ and $(G_2,P_2)$ are realizable, and let $Q_1$ and $Q_2$ be index-$p$ subgroups of $P_1$ and $P_2$ respectively. Assume that the local $P_1/Q_1$ and $P_2/Q_2$ Galois extensions are given by the Artin-Schreier polynomials $f_0=Z_1^p - Z_1 -f(x_0)$ and $g_0= Z_2^p-Z_2 -g(x_0)$ respectively, where $x_0$ is the local parameter of $\PP^1$ at $\infty$. Assume that $\ord_{x_0}(g)$ and $\ord_{x_0}(f)$ are different and not multiples of $p$. For $\alpha \in k$, let $M_{\alpha}/K$ and $N_{\alpha}/K$ be the $\mathbb{Z}/p$-Galois extensions given by the polynomials $f_{\alpha} = Z^p-Z- (1-\alpha)f(x_0) - \alpha g(x_0) - \alpha x_0^{-1} \in k((x_0))[Z]$ and $g_{\alpha}= Z^p -Z - (1+\alpha) g(x_0) + \alpha f(x_0) \in k((x_0))[Z]$, respectively.
Then there is a dense open subset $\mathcal{V}$ of $\mathbb{A}^1_t$ such that for all closed points $(t=\alpha)$ in $\mathcal{V}$, there exist a $P_1$-Galois extension $\widetilde M_{\alpha}/K$ and a $P_2$-Galois extension $\widetilde N_{\alpha}/K$ such that $M_{\alpha}=\widetilde M_{\alpha}^{Q_1}$, $N_{\alpha}=\widetilde M_{\alpha}^{Q_2}$, $\widetilde M_{\alpha}/K$ is realized by the pair $(G_1,P_1)$, and $\widetilde N_{\alpha}/K$ is realized by the pair $(G_2,P_2)$.
\end{lemma}

\begin{proof}
Let $U=\spec(k[[x_0]])$, $R=k[[t]]$. Let $A \coloneqq k((x_0))[t][Z]/(f_t)$, where $f_t(Z)=Z^p-Z- (1-t) f(x_0) -t g(x_0)- t x_0^{-1} \in k((x_0))[t][Z]$ and $B \coloneqq k((x_0))[t][Z]/(g_t)$, where $g_t(Z)=Z^p-Z - (1+t) g(x_0) + t f(x_0) \in k((x_0))[t][Z]$. We have the maps $\phi_{\A^1_t}' \colon S' \coloneqq \spec(A) \rightarrow \spec(k((x_0))[t])$ and $\psi_{\A^1_t}'\colon T' \coloneqq \spec(B) \rightarrow \spec(k((x_0))[t])$. Let $\phi_{\mathbb{A}^1_t} \colon S \rightarrow U\times_k \A^1_t$ and $\psi_{\mathbb{A}^1_t} \colon T \rightarrow U\times_k \A^1_t$ be the normalization maps in $A$ and $B$ respectively. Let $\phi_R \colon S_R \rightarrow U_R$ and $\psi_R \colon T_R\to U_R$ be their pullbacks under the map $U_R \rightarrow U \times_k \mathbb{A}^1_t$. Then the normalization of the closed fibre of $\phi_R$ and $\psi_R$ correspond to the field extensions $M_0/k((x_0))$ and $N_0/k((x_0))$ respectively. The covers $\phi_R$ and $\psi_R$ are branched only at the $R$-valued point $x_0=0$ since it is the only pole of the functions $f_t$ and $g_t$ in $R[[x_0]]$. By \cite[Theorem 3.11]{Harb_extn}, there exist connected $P_1$ and $P_2$-Galois \'{e}tale covers $\Phi_{\A^1_t}' \colon \widetilde{S}' \rightarrow \spec(k((x_0))[t])$ and $\Psi_{\A^1_t}' \colon \widetilde{T}' \rightarrow \spec(k((x_0))[t])$ dominating $\phi_{\A^1_t}'$ and $\psi_{\A^1_t}'$ respectively. Taking normalization of $U \times_k \A^1_t$ in the function fields of $\widetilde{S}'$ and $\widetilde{T}'$, we obtain $P_1$ and $P_2$-Galois covers $\Phi_{\mathbb{A}^1_t} \colon \widetilde S \rightarrow U\times_k \A^1_t$ and $\Psi_{\mathbb{A}^1_t} \colon \widetilde T \rightarrow U\times_k \A^1_t$ dominating $\phi_{\mathbb{A}^1_t}$ and $\psi_{\mathbb{A}^1_t}$ respectively. One also obtains $\Phi_R$ and $\Psi_R$ which dominate $\phi_R$ and $\psi_R$ by pull backs. 

So by Lemma \ref{lem_formal_patching}, there are $G_1$ and $G_2$-Galois covers of $\mathbb{P}^1_R$ satisfying the hypothesis of Lemma \ref{lem_dense_realization}. So there is a dense open subset $\mathcal{V}$ of $\mathbb{A}^1_t$ such that for all points $(t=\alpha)$ in $\mathcal{V}$, the extension $\widetilde M_{\alpha}/K$ is realized by the pair $(G_1,P_1)$, and the extension $\widetilde{N}_{\alpha}/K$ is realized by the pair $(G_2,P_2)$.
\end{proof}

\begin{theorem}\label{thm_pdt_g}
Let $G_1$, $G_2$ be two perfect quasi $p$-groups. Let $\tau\in G_1$ and $\sigma\in G_2$ be of order $p$ and $p^r$ for some $r$ respectively. Let $P_1 = \langle \tau \rangle \leq G_1$, $P_2 = \langle \sigma \rangle \leq G_2$. Assume that the pairs $(G_1,P_1)$ and $(G_2,P_2)$ are realizable. Then there exists $1 \leq a \leq p-1$ such that for $I \coloneqq \langle (\tau^a ,\sigma) \rangle \leq G_1 \times G_2$, the pair $(G_1 \times G_2, I)$ is also realizable.
\end{theorem}

\begin{proof}
Let $\phi_i \colon Y_i \rightarrow \mathbb{P}^1$ be a $G_i$-Galois cover of $\mathbb{P}^1$ branched only at $\infty$ with inertia group $P_i$ above $\infty$ and conductor $h_i$, for $i=1$, $2$. As usual it can be arranged that $h_2<h_1$ and the first upper jump of the $P_2$-extension is at least 2 (\cite[Theorem 2.2.2]{20}. Let $N_2$ be the index $p$ subgroup of $P_2$. Let the local $P_1$ and $P_2/N_2$ Galois extensions be given by the Artin-Schreier polynomials $f_0=Z_1^p - Z_1 -f(x_0) \in k((x_0))[Z_1]$ and $g_0= Z_2^p-Z_2 -g(x_0) \in k((x_0))[Z_2]$ respectively where $x_0$ is a local parameter at $\infty$. By Lemma \ref{lem_prep}, there is a dense open subset $\mathcal{V}$ of $\mathbb{A}^1_t$ such that for all points $(t=\alpha)$ in $\mathcal{V}$, the extension $M_{\alpha}/K$ given by the polynomial $f_{\alpha}=Z^p-Z- (1-\alpha) f(x_0) - \alpha g(x_0) - \alpha \cdot x_0^{-1}$ is realized by the pair $(G_1,P_1)$, and the extension $N_{\alpha}/K$ given by the polynomial $g_{\alpha}= Z^p -Z - (1+\alpha) g(x_0) + \alpha f(x_0)$ is dominated by a $P_2$-Galois extension $\widetilde N_{\alpha}/K$ which is realized by the pair $(G_2,P_2)$. So there is an $\alpha \neq 0, 1$ such that the points $(t=\alpha)$ and $(t=\alpha -1)$ both lie in $\mathcal{V}$. Let $X_i \rightarrow \mathbb{P}^1$ be the corresponding $G_i$-Galois covers of the affine line with inertia groups $P_i$ above $\infty$, and $\eta_i$ be points in $X_i$ over $\infty$ such that $K_{X_1,\eta_1}=M_{\alpha}$ and $K_{X_2,\eta_2}=\widetilde N_{\alpha-1}$.

Let $a_1 \in M_\alpha$ be a root of $f_\alpha$ and $a_2 \in N_{\alpha-1}$ be a root of $g_{\alpha-1}$. So $a_1-a_2$ is a root of the Artin-Schreier equation $Z^p-Z + \alpha  x_0^{-1}=0$. Since $-v_{x_0}(((1-\alpha) f(x_0) + \alpha g(x_0) + \alpha x_0^{-1})-(\alpha g(x_0) - (\alpha -1) f(x_0)))=1$, by Theorem \cite[Proposition 3.1]{15}, the compositum $M=M_{\alpha}\widetilde N_{\alpha}$ is a $Q = \mathbb{Z}/p \times \Z/p^r$-Galois extension with the first lower jump at $1$, and $Q_1=Q$, $Q_2=\text{Gal}(M/K(a_1-a_2)) \cong \mathbb{Z}/p^r$.

Let $X$ be the dominant connected component of the normalization of $X_1 \times_{\mathbb{P}^1} X_2$ containing the point $\eta:=(\eta_1,\eta_2)$. Then $\Theta \colon X \to \PP^1$ is a $G_1 \times G_2$-Galois cover branched only above $\infty$ with the local extension $K_{X,\eta}/K_{\PP^1,\infty}$ is given by $M/k((x_0))$, and the inertia subgroup $Q$ has a lower jump at $1$. Since $G_1$ and $G_2$ are perfect, $G_1 \times G_2$ is also perfect, and so by \cite[Theorem 3.7]{16}, there is a $G_1 \times G_2$-Galois cover of the affine line with inertia group above $\infty$ given by $I=\text{Gal}(M/K(a_1 - a_2))$ in $G$. Now the projection maps from $G_1\times G_2$ restricted to $I$ surjects onto $\Gal(M_{\alpha}/K)=P_1=\langle \tau \rangle$ and $\Gal(\widetilde N_{\alpha-1}/K)=P_2=\langle \sigma \rangle$. So $I$ is of the form $\langle (\tau^a,\sigma) \rangle$ for some $1 \leq a\leq p-1$.
\end{proof}

\begin{remark}\label{rmk3.9}
Observe that a finite product of groups preserves the properties of being quasi $p$ and perfect. In particular, if $G=A_{d_i}$, $d_i \geq p$, $i=1$, $2$, then the hypothesis of the Theorem \ref{thm_pdt_g} is satisfied.
\end{remark}

\begin{corollary}\label{cor_pdt}
Let $r\ge 2$ and $u \geq 1$ be integers. For $1 \leq i \leq u$, let $d_i \geq p$. Assume that the pairs $(A_{rp},\langle \tau \rangle)$ and $(A_{rp+1}, \langle \tau \rangle)$ are realizable where $\tau$ is the product of $r$ disjoint $p$-cycles in $A_{rp}$. Then the wild part of the inertia conjecture is true for $A_{d_1} \times \cdots \times A_{d_u}$.
\end{corollary}

\begin{proof}
In view of Corollary \ref{cor_A_d_containing_cycle}, Proposition \ref{prop_reduction_to_one_mod_p} and Remark \ref{rmk_general_IC_for_A_d}, the hypothesis implies that the wild part of inertia conjecture is true for Alternating groups. Set $G \coloneqq A_{d_1} \times \cdots \times A_{d_u}$. Let $P$ be a $p$-subgroup of $G$ whose conjugates generate $G$. There exist $g_1, \cdots, g_r \in P$ satisfying conditions (1)-(4) of Lemma \ref{lem_minimal_inertia_product}. 

The case $r=1$ will be proved by induction on $u$. Let $\pi_i \colon G \twoheadrightarrow A_{d_i}$ and $\pi \colon G \rightarrow A_{d_2} \times \cdots \times A_{d_u}$ be the projection maps. So there exists a cyclic subgroup $P'=\langle g_1\rangle$ of $P$ such that $\pi_1(P')=\langle \tau\rangle\le A_{d_1}$ is a cyclic group of order $p$, $\pi(P')$ is a cyclic $p$-group with generator $\sigma$, say, and $\pi_i(P')$ are nontrivial subgroups for all $i$. By induction hypothesis on $u$, $(A_{d_2}\times \cdots A_{d_u}, \pi(P'))$ is realizable. Moreover $P'=\langle (\tau,\sigma) \rangle$ and $\tau$ is of order $p$. Now by Theorem \ref{thm_pdt_g}, $(A_{d_1} \times \cdots \times A_{d_u},I)$ is realizable where $I=\langle (\tau^a,\sigma) \rangle$ for some $1 \leq a\leq p-1$. But there is an automorphism of $A_{d_1}$ which sends $\tau$ to $\tau^a$. Hence $(A_{d_1} \times \cdots \times A_{d_u}, P')$ is realizable. Finally by Harbater's result \cite[Theorem 2]{2}, $(A_{d_1} \times \cdots \times A_{d_u},P)$ is realizable.

Now for $r\ge 2$, in the Notation \ref{not_pdt} of Lemma \ref{lem_minimal_inertia_product}, let $H_i=H_{S(g_i)}$. By $r=1$ case, the pairs $(H_i,\langle g_i\rangle)$ are realizable for $1\le i\le r$. Now the result follows from Theorem \ref{thm_Raynaud}.
\end{proof}

\begin{corollary}\label{cor_pdt_less_than_2p}
Let $u \geq 1$ be an integer and let $p \leq d_i < 2p$ for all $1 \leq i \leq u$. The wild part of the inertia conjecture is true for $A_{d_1} \times \cdots \times A_{d_u}$. Moreover, if $Q$ is any $p$-group, the wild part of the inertia conjecture is true for $A_{d_1} \times \cdots \times A_{d_u} \times Q$.
\end{corollary}

\begin{proof}
When $d \leq 2p-1$, any $p$-subgroup of $A_d$ contains a $p$-cycle. So the first part of the statement follows from Corollary \ref{cor_A_d_containing_cycle} and Corollary \ref{cor_pdt}. The second statement follows from the first one and \cite[Corollary 4.6]{15}.
\end{proof}

\section{A New Equation for Alternating Group covers}\label{sec_neweq}
Let $p$ be an odd prime and let $a \geq 3$ be an integer. Set $d=a p$. In this section we produce an $A_d$-Galois \'{e}tale cover of the affine line with wild part of the inertia group generated by a $p$-cycle which occurs as the Galois closure of a degree-$d$ cover given by an explicit affine equation. Since there are ways to increase the upper jump of the ramification filtration, an interesting problem is to find covers with minimal possible upper jump. Riemann-Hurwitz formula provides a lower bound for the upper jump. We show that the resulting cover attains this lower bound.

\begin{lemma}\label{lem_num}
Let $p>2$ be a prime number, $d=a p$, $a \geq 3$, $j \coloneqq \frac{p-1}{(p-1,d+1)}$. Then there exists an integer $s$, $1 \leq s \leq p-1$, such that $s|j$ and $(j,d-p-s)=1$.
\end{lemma}

\begin{proof}
Since $j|(p-1)$, for any integer $s$ we have $(j,d-p-s)=(j,a-1-s)$. When $a=3$, for any $p>2$, $s=1$ satisfies the stated properties. For some $a$ and $p$, if $j=1$, then $s=1$ is again a solution. Now assume that $s$ satisfies the stated properties for some $p$ and $a$. Consider $a'=a+p-1$, $d'=a' p$, and $j' = \frac{p-1}{(p-1,(a+p-1) p +1)}$. Then $j'= \frac{p-1}{(p-1,a+1)}=j$ and
$$(j',d'-p-s)=(j,(a+p-2) p -s) = (j, (a-1) p -s)=(j,d-p-s).$$
Since $(j,d-p-s)=1$ by assumption, $s$ is also a solution for $p$ and $a'$.

So for a fixed $p>2$, we may assume that $4 \leq a \leq p+1$. Let $Q_1$ and $Q_2$ be the set of primes dividing $a+1$ and $a-1$ respectively. Let $Q$ be the set of primes dividing $p-1$, and for any prime $q \in Q$, let $\lambda_q$ be the highest positive integer such that $q^{\lambda_q}|(p-1)$. Let $I_1$ be the set of primes dividing $(p-1,a+1)$ and $I_2$ be the set of primes dividing $(p-1,a-1)$. Set $I \coloneqq I_1 \cup I_2$. Let $(p-1,a+1)= \Pi_{q \in I_1} q^{m_q}$ and $a-1=\Pi_{q \in Q_2} q^{s_q}$ (we use the notion that an empty product equals $1$). For $q \in I_1 \setminus I_2$, set $\alpha_q \coloneqq m_q$ and for $q \in I_2$, set $\alpha_q \coloneqq \lambda_q$. So we have
$$j = \frac{p-1}{(p-1,d+1)} = \frac{p-1}{(p-1,a+1)} = \frac{p-1}{\Pi_{q \in I_1} q^{m_q}}.$$
Let $s = \frac{p-1}{\Pi_{q \in I} q^{\alpha_q}}$ and $v = \frac{\Pi_{q \in I} q^{\alpha_q}}{\Pi_{q \in I_1} q^{m_q}}$. Note that $s v = j$. Let
$$g = (j , d-p-s) = (j,a-1-s) = (\frac{p-1}{\Pi_{q \in I_1} q^{m_q}} , \Pi_{q \in Q_2} q^{s_q} - \frac{p-1}{\Pi_{q \in I} q^{\alpha_q}}).$$
Let $q$ be a prime dividing $g$. So $q| j = s v$. If $q|s$, then $q \not\in I_2$ as $\alpha_q = \lambda_q$ for $q \in I_2$. Since $q$ divides both $j$ and $s$, it must divide $a-1$, and so $q \in I_2$, a contradiction. So $q \nmid s$ and we have $q|v$. So $q \in I_2$, and thus $q|(a-1)$. Thus $q \nmid (a-1-s)$ contradicting $q$ divides $g$. So we have $g=1$.
\end{proof}

\begin{remark}\label{rmk_num}
Let $p$, $a$ and $s$ be as in the above Lemma \ref{lem_num}. When $a=3$, we have $s=1$, and so $d-p-s=2 p -1 > \frac{3p}{2}$. For $a \geq 4$, we have $s \leq p-1$, and so $d-p-s \geq d-p - (p-1) > \frac{d}{2}$. \end{remark}

Now set $j \coloneqq \frac{p-1}{(p-1,d+1)}$. Using Lemma \ref{lem_num}, fix an integer $1 \leq s \leq p-1$, such that $s|j$ and $(j,d-p-s)=1$. Consider the degree-$d$ cover $\psi: Y \rightarrow \mathbb{P}^1$ given by the affine equation
\begin{equation}\label{eq4}
x = \frac{1+ y^{d-s} (y+1)^s}{ y^{d-p-s} (y+1)^s}.
\end{equation}
Put $h(y) \coloneqq 1 + y^{d-s} (y+1)^s \in k[y]$. Writing the above equation as
$$f(x,y) = h(y) - x y^{d-p-s} (y+1)^s=0,$$
its $y$-derivative is given by
$$f_{y}(x,y) = h'(y) + s x y^{d-p-s-1} (y+1)^{s-1}.$$
Assume that $f$ and $f_y$ have a common zero $(a,b)$. So $f_y(a,b)=0$ which implies that $h'(b) = - s a b^{d-p-s-1} (b+1)^{s-1}$. Also since $f(a,b)=0$, we have $s h(b) + b (b+1) h'(b)=0$. But we see that $s h(y)+ y (y+1) h'(y)= s \neq 0$ in $k[y]$, showing that $f$ and $f_y$ cannot have a common zero. Since the cover $\psi$ is non-trivial, it is branched only above $x = \infty$. Observe that the Equation \ref{eq4} can be written in the following manner.
\begin{eqnarray*}
x^{-1} & = & \frac{y^{d-p-s} (y+1)^s}{ 1 + y^{d-s} (y+1)^s};\\
x^{-1} & = & \frac{((y+1)-1)^{d-p-s} (y+1)^s}{1 + ((y+1)-1)^{d-s} (y+1)^s};\\
x^{-1} & = & \frac{y^{-p} (y^{-1}+1)^{s}}{ y^{-d} + (y^{-1}+1)^s}.
\end{eqnarray*}
Thus $v_{(y)}(x^{-1})=d-p-s$, $v_{(y+1)}(x^{-1})=s$, and $v_{(y^{-1})}(x^{-1})= p$. So there are exactly three points in $Y$ in the fibre of $\psi$ above $x=\infty$ with ramification indices $p$, $s$ and $d-p-s$. Let $\phi: \widetilde{Y} \rightarrow \mathbb{P}^1$ be the Galois closure of $\psi$ with Galois group $G$. Since $0$ and $-1$ are the only roots of $y^{d-p-s}(y+1)^s$ and $h(0)=1=h(-1)$, by Lemma \cite[Section 19, First Irreducibility Lemma]{7}, $f(x,y)$ is irreducible in $k(x)[y]$. So $G$ is a transitive quasi $p$ subgroup of $S_d=\Sym(\text{zeros of } f)$.

\begin{proposition}\label{prop3.1}
Let $p$ be an odd prime, $a \geq 3$ be an integer, and let $d=a p$. Let $\psi \colon Y \to \PP^1$ be the degree-$d$ cover given by affine Equation \ref{eq4}. Let $\phi : \widetilde{Y} \rightarrow \mathbb{P}^1$ be its Galois closure. Then $\phi$ is an $A_d$-Galois cover of $\mathbb{P}^1$ branched only at $\infty$, with inertia group $I=\langle \tau \rangle \rtimes \langle \beta \rangle$ above $\infty$ with $\tau$ a $p$-cycle, and $\beta$ of order $\text{l.c.m.}(\frac{p-1}{(p-1,d+1)},\text{l.c.m.}(s,d-p-s))$. Furthermore, it has upper jump $h/m=(d+1)/(p-1)$.
\end{proposition}

\begin{proof}
Let $x_0=x^{-1}$, and consider the local equation of $Y$ near a preimage of $x=\infty$. Then $I$ is the Galois group of the splitting field of $f(y)= x_0 (1 + y^{d-s} (y+1)^s)- y^{d-p-s} (y+1)^s$ over $k((x_0))$. By Hensel's Lemma, $f(y)=g(y) h_1(y) h_2(y)$ in $k[[x_0]][y]$ with $g(y) \equiv -1 (\text{ mod } x_0)$ of degree $p$, $h_1(y) \equiv y^{d-p-s} (\text{ mod } x_0)$ of degree $d-p-s$, and $h_2(y) \equiv (y+1)^s (\text{ mod } x_0)$ of degree $s\le p-1$. An argument similar to Proposition \ref{prop_3.3B} shows that $h_1(y)$ is Eisenstein of degree $d-p-s$. Hence by Lemma \ref{lem3.3}, $\text{Gal}(h_1 ,k((x_0))) =\mu_{d-p-s}$ and $p \nmid |\Gal(h_2(y),k((x_0)))|$. Also $p$ divides $|I|$ and $g$ has degree $p$. So $\text{Gal}(g, k((x_0)))$ is a transitive subgroup of some $S_p=\Sym(\text{zeros of $g$})$. Hence $\text{Gal}(g, k((x_0))) \leq S_p \cap G$. So the Sylow $p$-subgroup of $I$ is the same as the Sylow $p$-subgroup of $\text{Gal}(g, k((x_0)))$, and it is generated by a $p$-cycle $\tau$. Without loss of generality, we may assume that $\tau = (1,\cdots,p)$.

By Proposition \ref{prop2.1}, the inertia group $I = \langle \tau \rangle \rtimes \langle \beta \rangle$, where $\beta = \theta^i \cdot \omega$ for some element $\theta$ in $S_p$ of order $(p-1)$, integer $0 \leq i \leq p-1$, and $\omega \in H =\text{Sym}(\{p+1, \cdots ,d\})$. Let $h$ be the conductor of the cover $\phi$, and let $\text{ord}(\beta) = m$. Since there are exactly three points in $\psi^{-1}(\infty)$ with ramification indices $p$, $s$ and $d-p-s$, by Lemma \ref{lem2.2}, $\omega$ is of the form $\omega= \omega_1 \cdot \omega_2$, where $\omega_1$ is an $s$-cycle and $\omega_2$ is a $(d-p-s)$-cycle disjoint from $\omega_1$. Now the Equations \ref{eq1}, \ref{eq2} and \ref{eq3} provide $\frac{h}{m}=\frac{d+1}{p-1}$, $\text{ord}(\theta^i)=m''=\frac{p-1}{(p-1,d+1)}$, $m'=\frac{\text{l.c.m.}(s,d-p-s)}{(m'', \text{l.c.m.}(s,d-p-s))}$. Thus $m''$ equals $j$ and $\beta$ has order $m'm''=\text{l.c.m.}(\frac{p-1}{(p-1,d+1)},\text{l.c.m.}(s,d-p-s))$. As $\theta^i$ has order $m''=j$, $s|j$ and $(j,d-p-s)=1$, $\beta^j=\omega^j$ is a $(d-p-s)$-cycle. By Remark \ref{rmk_num}, we have $d-p-s> d/2$. So the cycle $\beta^j$ fixes less than $d/2$ points, and hence $G$ is primitive (\cite[Remark 1.6]{24}). So by Jordan's Theorem \cite[Theorem 3.3.E]{14}, $G=A_d$.
\end{proof}

\begin{corollary}\label{cor_min}
Let $p$ be an odd prime, $a \geq 3$ be an integer, and let $d =a p$. The $A_d$-Galois cover $\phi \colon \widetilde{Y} \to \mathbb{P}^1$ has the minimum possible upper jump among all the $A_d$-Galois covers of $\mathbb{P}^1$ \'{e}tale away from $\infty$ for which the Sylow $p$-subgroup of the inertia group above $\infty$ is generated by a $p$-cycle.
\end{corollary}

\begin{proof}
As a consequence of the Riemann-Hurwitz formula (Equation \ref{eq1}), for any transitive subgroup $G$ of $S_d$, any $G$-Galois covers of $\mathbb{P}^1$ \'{e}tale away from $\infty$ for which the Sylow $p$-subgroup of the inertia group above $\infty$ is generated by a $p$-cycle has upper jump $\geq \frac{d+1}{(p-1)}$. Now the result follows from Proposition \ref{prop3.1}.
\end{proof}

\bibliographystyle{amsplain}

\end{document}